\documentclass[envcountsame,
 ]
 {llncs}

\usepackage[table]{xcolor}
\usepackage{latexsym,amssymb,amsmath,amsfonts}
\usepackage[cmtip,arrow]{xy}
\usepackage{pb-diagram,pb-xy}
\usepackage{graphicx}
\usepackage{tikz}
\usepackage{multicol}
\usepackage{multirow}
\usepackage{enumerate}
\usepackage{hyperref}
\usepackage{verbatim}
\usepackage[utf8]{inputenc}

\usetikzlibrary{matrix,arrows}

\spnewtheorem{Case}{Case}{\bfseries}{\itshape}
\newtheorem{subcase}{Subcase}

\author{Carlos A. Alfaro\inst{1}
\thanks	{Supported by SNI and CONACYT grant 166059.}
\and Alan Arroyo\inst{2}
\thanks{Suported by CONACYT.}
\and Marek Der\v n\' ar\inst{3}
\and Bojan Mohar\inst{4}
\thanks{Supported in part by an NSERC Discovery Grant (Canada), by the Canada Research Chairs program, and by a Research Grant of ARRS (Slovenia). 
On leave from: IMFM, Department of Mathematics, Ljubljana, Slovenia.}}

\institute{Banco de M\'exico, Ciudad de M\'exico, M\'exico
\and
Department of Combinatorics and Optimization,\\
University of Waterloo,  Canada
\and
Faculty of Informatics,
Masaryk University,
Brno, Czech Republic
\and
Department of Mathematics,
Simon Fraser University,
Burnaby, 
Canada
}

\newcommand\CR{\mathop{cr}}
\date{\today}

\begin{document}


	\title{The crossing number of the cone of a graph}

	\maketitle

\begin{abstract}
		Motivated by a problem asked by Richter and by the long standing Harary-Hill conjecture, we study the relation between the crossing number of a graph $G$ and the crossing number of its cone $CG$, the graph obtained from $G$ by adding a new vertex adjacent to all the vertices in $G$.
		Simple  examples show that the difference $cr(CG)-cr(G)$ can be arbitrarily large for any fixed $k=cr(G)$.
		In this work, we are interested in finding the smallest possible difference, that is, for each non-negative integer $k$, find the smallest $f(k)$ for which there exists a graph with crossing number at least $k$ and cone with crossing number $f(k)$.
		For small values of $k$, we give exact values of $f(k)$ when the problem is restricted to simple graphs, and show that $f(k)=k+\Theta (\sqrt {k})$ when multiple edges are allowed.
\end{abstract}

\section{Introduction}	

Little is known on the relation between the crossing number and the chromatic number. In this sense Albertson's conjecture (see \cite{ACF}), that if $\chi(G)\geq r$, then $cr(G)\geq cr(K_r)$, has taken a great interest.
Albertson's conjecture has been proved \cite{ACF,BT,OZ} for $r\leq16$.
It is related to Haj\'os' Conjecture that every $r$-chromatic graph contains a subdivision of $K_r$.
If $G$ contains a subdivision of $K_r$, then $cr(G)\geq cr(K_r)$.
Thus Albertson's conjecture is weaker than Haj\'os' conjecture, however Haj\'os' conjecture is false for any $r\geq 7$ \cite{C}.

The \emph{cone} of a graph $G$ is the graph $CG$ obtained from $G$ by adding an \emph{apex}, a new vertex that is adjacent to each vertex in $G$.
Many properties of a graph are automatically transferred to its cone.  For example, if $G$ is $r$-coloring-critical, then $CG$ is $(r+1)$-coloring-critical.
During the Crossing Numbers Workshop in 2013, in an attempt to understand Alberston's conjecture, Richter proposed the following problem: Given an integer $n\geq 5$ and a graph $G$ with crossing number at least $cr(K_{n})$, does it follow that the crossing number of its cone $CG$ is at least $cr(K_{n+1})$?
There are examples where these two values can differ arbitrarily (for instance, if $G$ is the disjoint union of $K_{4}$'s and $K_{5}$'s).
What is less clear is how close these values can be.

The answer to Richter's question is positive for the first interesting case when $n=5$:  Kuratowski's theorem implies that the cone of any graph with crossing number at least $cr(K_{5})=1$ contains a subdivision of $CK_{5}$ or
$CK_{3,3}$, and each of these graphs has crossing number at least $cr(K_{6})=3$.
Unfortunately, the answer is negative for the next case, as the graph in Figure \ref{figure:counterexample1} shows.
This graph has crossing number 3, and a cone with crossing number at most $6$, and this is less than $cr(K_{7})=9$.
This motivated us to investigate the following question.

\begin{problem}\label{problem:variation}
	For each $k\geq 0$, find the smallest integer $f(k)$ for which there is a graph $G$ with crossing number at least $k$ and its cone has  $cr(CG)=f(k)$.
\end{problem}

\begin{figure}[htb]
		\begin{center}
				\begin{tikzpicture}[scale=1]
				\tikzstyle{every node}=[minimum width=0pt, inner sep=1pt, circle]
					\draw (0,0) circle (1);
					\draw (270:.35) node (3) [draw,fill=white] {};
					\draw (30:.35) node (1) [draw,fill=white] {};
					\draw (150:.35) node (2) [draw,fill=white] {};
					\draw (0:1) node (32) [draw,fill=white] {};
					\draw (60:1) node (11) [draw,fill=white] {};
					\draw (120:1) node (12) [draw,fill=white] {};
					\draw (180:1) node (21) [draw,fill=white] {};
					\draw (240:1) node (22) [draw,fill=white] {};
					\draw (300:1) node (31) [draw,fill=white] {};
					\draw (12) -- (1) -- (11) -- (2) -- (1) -- (32) -- (3) -- (31) -- (1) -- (3) -- (22) -- (2) -- (21) -- (3) -- (2) -- (12);
				\end{tikzpicture}
		\end{center}
		\caption{A counterexample to Richter's question when $n=6$.}
		\label{figure:counterexample1}
	\end{figure}
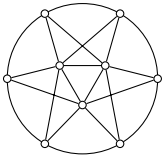

Note that $f(k)$ can also be defined as the largest integer such that every graph with $cr(G)\geq k$, has $cr(CG)\geq f(k)$.
An upper bound to the function $f(k)$ is obtained from the graph in Figure \ref{figure:counterexample1}, by changing the multiplicity of each edge to $r$.
Any drawing of the new graph has at least $3r^{2}$ crossings, and its cone has crossing number $3r^{2}+3r$.
This shows that $f(k)\leq k+\sqrt{3k}$.
Our main result shows that this is close to be best possible.

\begin{theorem}\label{thm:lowerbound}
		Let $G$ be a graph with $cr(G)\geq k$. Then
		$ cr(CG)\geq k+\sqrt{k/2}$.
\end{theorem}

Thus we have the following:

\begin{corollary}
	For multigraphs we have $f(k) = k + \Theta(\sqrt{k}\,)$.
\end{corollary}

The paper is organized as follows. Page drawings, a concept intimately related to  drawings of the cone of a graph, are defined in Section \ref{Sec2} and used throughout the subsequent sections. Although, there seems to be a connection between 1-page drawings and drawings of the cone, their exact relationship is much more subtle. Our proofs are instructive in this manner and provide further understanding of these concepts. 

The proof of our main result, Theorem \ref{thm:lowerbound} is provided in Section \ref{Sec3}.
In Section \ref{Sec4}, we restrict Problem \ref{problem:variation} to the case of simple graphs. 
To distinguish between these two  problems we use $f_{s}(k)$ instead of $f(k)$.  Along this paper, a graph is allowed to have multiple edges but no loops; when  our graphs  have no multiple edges, then we refer   them as simple graphs.
We find the smallest values of $f_s$ by showing that $f_s(1)=3$, $f_s(2)=5$, $f_s(3)=6$, $f_s(4)=8$  and $f_s(5)=10$. These initial values may suggest that $f_s(k)\ge 2k$. However, in Section \ref{sect:asymptotics} we show that
$$
  f_s(k)=k+o(k),
$$
and provide additional justification for a more specific conjecture that
$$
  f_s(k) = k + \sqrt{2}\,k^{3/4}(1+o(1)).
$$

\section{Page drawings}\label{Sec2}

In this section we describe a perspective provided from  considering {\em page drawings} of graphs, a concept that has  been studied in its own and has interesting applications. The relation between 1- and 2-page drawings has shown to be handy as it is used in the proofs of Theorems \ref{thm:lowerbound} and \ref{thm:cr(CG)geq3}. A more detailed  discussion on the relevant aspects of this section  can be found  in \cite{BE14,BZ,KPS}.

 For an integer  $k\geq 1$, a {\em $k$-page book} consists  of $k$ half planes sharing their boundary line $\ell$ (spine).
A {\em $k$-page-drawing} is a drawing of
a graph in which vertices are placed in the spine of a $k$-page book, and each edge arc
is contained in one page. A convenient way to visualize a $k$-page drawing is by means of the {\em circular model}.
In this model each page is represented by a  unit 2-dimensional disk, so that the vertices are arranged identically on
each disk boundary and each edge is drawn entirely in exactly one  disk. In this work we are   only
interested in  $1$ and $2$-page drawings, and, to be more precise, in the following problem.

\begin{problem}
Given a $1$-page drawing of a graph $G$ with $k$ crossings, find an upper bound on
the number of crossings of an optimal 2-page drawing of $G$ while having the order of vertices of $G$ on the spine unchanged.
\end{problem}

In other words, if the drawing of $G$ in the plane is such that all the vertices are incident to the  outer-face (which is equivalent to having a 1-page drawing),  what is the most efficient way to redraw some edges in the outer-face to reduce the number of crossings?
For this purpose, we define the {\em circle graph} $C_{D}$ of   any $1$-page drawing $D$ of $G$ as the graph whose vertices are the edges of $G$, and any two elements are adjacent if they cross in $D$. Note that $C_D$ depends only on the cyclic order of the vertices of $G$ in the spine.

A related problem was previously formulated by Kainen in \cite{Kainen}, where he studied the {\em outerplanar crossing number} of a graph as the minimum number of crossings in any drawing of $G$ so that all its vertices are incident to the same face.
Clearly, the crossing number of $CG$ is at most the outer-planar crossing number of $G$. 	
Although,  Kainen was interested in finding an $n$-vertex  graph that has the largest difference between its crossing number and its outer-planar crossing number, for us it will be useful to consider drawings in which the vertices are incident to the same face.

Turning a 1-page drawing into a 2-page drawing is equivalent to finding a bipartition $(X,V(C_{D})\setminus X)$ of the vertices of $C_{D}$, each part representing the set of edges of $G$ drawn in one of the pages.
Minimizing the number of crossings in the obtained 2-page drawing of $G$ is equivalent  to maximize  the  number edges in $C_{D}$  between $X$ and $V(C_{D})\setminus X$. This last problem is known as  the {\em max-cut problem}, and if the considered graph $C_D$ has $m$ edges, then, a well-known result of Erd\H{o}s \cite{Erdosmaxcut} states that its maximum edge-cut has size more than $m/2$.  Improvements to this general bound are known (see \cite{Ed1}, \cite{Ed2} and a more recent survey \cite{BS}). For our purpose the following bound of Edwards will be useful.

\begin{lemma}[Edwards \cite{Ed1,Ed2}]
\label{lem:maxcut Edwards}
Suppose that $G$ is a graph of order $n$ with $m\geq 1$ edges. Then $G$ contains a bipartite subgraph with at least $\frac{1}{2}m + \sqrt{\frac{1}{8}m+\frac{1}{64}} - \frac{1}{8} > \frac{1}{2}m$ edges.
\end{lemma}

In our context, this  result translates to the following observation that we will use.

\begin{corollary}\label{cor:1-2-page}
Let $D$ be a 1-page drawing of a graph $G$ with $k\ge1$ crossings.
Then some edges of $G$ can be redrawn in a new page, obtaining a 2-page drawing with at most $\frac{1}{2}k - \sqrt{\frac{1}{8}k+\frac{1}{64}} + \frac{1}{8}$ crossings.
Such a drawing can be found in time $O(|E(G)|+k)$.
\end{corollary}
The proof of Corollary \ref{cor:1-2-page} will be provided in the full version.

\section{Lower bound on the crossing number of the cone} \label{Sec3}

This section contains the proof of our main result.

\begin{proof}[of Theorem \ref{thm:lowerbound}]
Let $\widehat{D}$ be an optimal drawing of  the cone $CG$ of $G$ with apex $a$, and suppose $\widehat{D}$ has less than $ k+\sqrt{ k/2}$ crossings.
We   consider  $D=\widehat{D}|_{G}$, the drawing of  $G$ induced by $\widehat{D}$. If we let  $t$ to be the number
of crossings in $D$, then we have
\begin{equation} \label{equation1}
k\leq t< k+\sqrt{k/2}\text{.}
\end{equation}

For each vertex $v\in V(G)\cup\{a\}$, let $s_{v}$ be the number of crossings in $\widehat{D}$ involving edges incident with $v$.
Using that $cr(\widehat{D})<k+\sqrt{ k/2}$ and the left-hand side inequality in (\ref{equation1}), we obtain that $s_{a}<\sqrt{ k/2}$.

Consider $x_{1}$,\ldots,$x_{s_{a}}$, the crossings involving edges incident with $a$.
Since  $\widehat{D}$ is optimal,  each of these crossings is between an edge incident to $a$ and an edge in $G$.
Let $e_{1}$,\ldots,$e_{s_{a}}$ be the
list of edges in $G$ (we allow repetitions) so that $x_{i}$ is the crossing between $e_{i}$ and an edge incident with $a$.
We subdivide  each edge $e_{i}$ in $D$ using two points close to the crossing $x_{i}$, and we remove the edge segment $\sigma_{i}$ joining
these   new two vertices, in order to obtain a drawing $D_{0}$ of a graph $G_{0}$ with $t$ crossings (see Figure \ref{figure:drawingOfD_0}).

	\begin{figure}[h!]
		\begin{center}
			\begin{tabular}{c@{\extracolsep{2cm}}c@{\extracolsep{2cm}}c}
				\begin{tikzpicture}[scale=.9]
				\tikzstyle{every node}=[minimum width=0pt, inner sep=1pt, circle]
					\draw (30:2) node (0) [draw, fill=white, label = right:{\small $a$}] {};
					\draw (0:.35) node (1) [draw,fill=white] {};
					\draw (90:.35) node (2) [draw,fill=white] {};
					\draw (180:.35) node (3) [draw,fill=white] {};
					\draw (270:.35) node (4) [draw,fill=white] {};
					\draw (0:1) node (11) [draw,fill=white] {};
					\draw (60:1) node (12) [draw,fill=white] {};
					\draw (120:1) node (13) [draw,fill=white] {};
					\draw (180:1) node (14) [draw,fill=white] {};
					\draw (240:1) node (15) [draw,fill=white] {};
					\draw (300:1) node (16) [draw,fill=white] {};
					\draw (0) -- (1)
					(12) -- (0) -- (11);
					\draw (0) edge[bend right] (13)
					(0) edge[bend left] (16)
					(0) .. controls (360:2) and (300:2) .. (15)
					(0) .. controls (60:2.5) and (120:2.5) .. (14)
					(0) .. controls (90:2.5) and (150:1.5) .. (3);
					\draw (0) .. controls (0:1.5) and (330:1) .. (4);
					\draw (0) .. controls (60:1.5) and (100:1) .. (2);
					\draw (11) -- (12) -- (13) -- (14) -- (15) -- (16) -- (11);
					\draw (13) -- (2) -- (12) -- (1) -- (11);
					\draw (1) edge[bend left] (15);
					\draw (15) -- (3) -- (4) -- (2) -- (1) -- (4) -- (16);
					\draw (2) -- (3) -- (14) -- (4);
					\draw (14) -- (2);
					\draw (315:1.8) node (17) {\small $\widehat{D}$};
				\end{tikzpicture}
				&
				\hspace{-0.4in}
				\begin{tikzpicture}[scale=1.5]
				\tikzstyle{every node}=[minimum width=0pt, inner sep=1pt, circle]
					\draw (0:.35) node (1) [draw,fill=white] {};
					\draw (90:.35) node (2) [draw,fill=white] {};
					\draw (180:.35) node (3) [draw,fill=white] {};
					\draw (270:.35) node (4) [draw,fill=white] {};
					\draw (0:1) node (11) [draw,fill=white] {};
					\draw (60:1) node (12) [draw,fill=white] {};
					\draw (120:1) node (13) [draw,fill=white] {};
					\draw (180:1) node (14) [draw,fill=white] {};
					\draw (240:1) node (15) [draw,fill=white] {};
					\draw (300:1) node (16) [draw,fill=white] {};
					\draw[thick] (16) -- (11) -- (12) -- (13) -- (14);
					\draw (14) -- (15) -- (16);
					\draw (13) -- (2) -- (12) -- (1) -- (11);
					\draw (1) edge[bend left, thick] (15);
					\draw (15) -- (3) -- (4) -- (2) -- (1) -- (4) -- (16);
					\draw (2) -- (3) -- (14) -- (4);
					\draw (14) edge[thick] (2);
					\draw (330:1.1) node (e1) {\small $e_1$};
					\draw (30:1.1) node (e2) {\small $e_2$};
					\draw (90:1.1) node (e3) {\small $e_3$};
					\draw (150:1.1) node (e4) {\small $e_4$};
					\draw (148:.6) node (e5) {\small $e_5$};
					\draw (275:.75) node (e6) {\small $e_6$};
					\draw (220:1.25) node (e7) {\small $D$};
				\end{tikzpicture}
				&
				\hspace{-0.4in}
				\begin{tikzpicture}[scale=1.5]
				\tikzstyle{every node}=[minimum width=0pt, inner sep=1pt, circle]
					\draw (0:.35) node (1) [draw,fill=white] {};
					\draw (90:.35) node (2) [draw,fill=white] {};
					\draw (180:.35) node (3) [draw,fill=white] {};
					\draw (270:.35) node (4) [draw,fill=white] {};
					\draw (0:1) node (11) [draw,fill=white] {};
					\draw (60:1) node (12) [draw,fill=white] {};
					\draw (120:1) node (13) [draw,fill=white] {};
					\draw (180:1) node (14) [draw,fill=white] {};
					\draw (240:1) node (15) [draw,fill=white] {};
					\draw (300:1) node (16) [draw,fill=white] {};
					\draw[thick] (16) -- (11) -- (12) -- (13) -- (14);
					\draw (14) -- (15) -- (16);
					\draw (13) -- (2) -- (12) -- (1) -- (11);
					\draw (1) edge[bend left, thick] (15);
					\draw (15) -- (3) -- (4) -- (2) -- (1) -- (4) -- (16);
					\draw (2) -- (3) -- (14) -- (4);
					\draw (14) edge[thick] (2);
					
					\draw (330:.88) node (e11) [fill=white, minimum width=12pt] {};
					\draw (298:.42) node (e12) [fill=white, minimum width=5pt] {};
					\draw (4) edge[dotted] (340:.85) edge[dotted] (320:.85);
					
					\draw (30:.88) node (e2) [fill=white, minimum width=12pt] {};
					\draw (1) edge[dotted] (40:.85) edge[dotted] (20:.85);
					
					\draw (90:.88) node (e3) [fill=white, minimum width=12pt] {};
					\draw (2) edge[dotted] (100:.85) edge[dotted] (80:.85);
					
					\draw (150:.88) node (e4) [fill=white, minimum width=12pt] {};
					\draw (160:.55) node (e5) [fill=white, minimum width=7pt] {};
					\draw (3) edge[dotted] (160:.85) edge[dotted] (140:.85);
					\draw (220:1.25) node (e6) {\small $D_0$};
				\end{tikzpicture}\\[2mm]
				(a) & \hspace{-0.4in} (b) & \hspace{-0.3in} (c)\\
			\end{tabular}
		\end{center}
		\caption{A drawing where the crossed edges are cut.
		}
		\label{figure:drawingOfD_0}
	\end{figure}
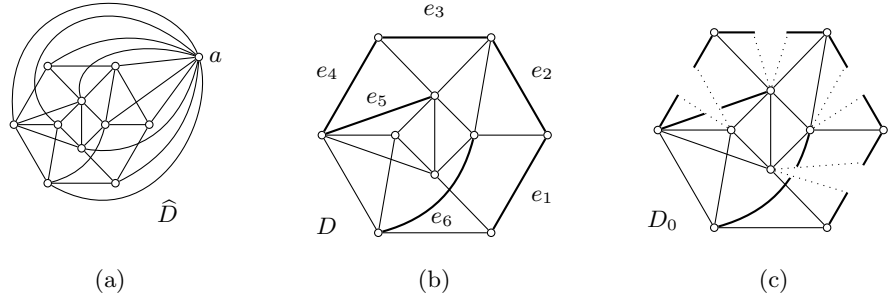

The obtained drawing $D_0$ has all its vertices incident to the face of $D_{0}$ containing the point corresponding to the apex vertex $a$ of $CG$ in $\widehat{D}$.
For simplicity, we may assume that this is the unbounded face of $D_{0}$. It follows that there exists a simple closed curve $\ell$ in the closure of this face, containing all the vertices of $G_{0}$.
Thus, $D_{0}$ gives rise to a $1$-page drawing of $G_{0}$ with spine $\ell$.

Now construct a new drawing of $G$ as follows:
	\begin{enumerate}
		\item
		Start with the 1-page drawing $D_0$. Partition the edges of $G_{0}$ according to Corollary \ref{cor:1-2-page}, and draw
			the edges  of one part in page 2 outside $\ell$.
		\item Reinsert edge segments $\sigma_{1},\ldots,\sigma_{s_{a}}$ as they where drawn in $D$, to obtain a drawing $D_{1}$ (of a subdivision) of $G$. These segments do not cross each other, but they may cross some of the edges of $G_0$ that we placed in page 2 in step 1. 
	\end{enumerate}	
	
Now we  estimate the number of crossings in $D_{1}$. According to Corollary \ref{cor:1-2-page}, after step 1 we
obtain a 2-page drawing $D_{0}$ with less than $t/2-\sqrt{t/8}+1/8$ crossings. After step 2 we gain some new crossings between the added segments $\sigma_{1},\ldots,\sigma_{s_{a}}$ and the edges  of $G_0$ drawn on page 2 in step 1. 

	\begin{claim}\label{lemma:estimatednumbercrossings}
		The number of new crossings between $\sigma_{1},\dots,\sigma_{s_{a}}$ and the edges drawn on page 2 in step 1 is at most $(k-1)/2$.
	\end{claim}

	\begin{proof}
		We may assume that, for each $v\in V(G)$, $s_{v}< \sqrt{ k/2}$. Otherwise, by
	removing $v$ and all the edges incident to $v$, we obtain  a drawing of $CG-v$ containing a
	subdrawing of $G$, in which $v$ is represented as the apex, and this drawing has less than $k$ crossings,  a contradiction.
    
    Let $e\in E(G)$ be an edge having ends $u$, $v\in V(G)$. Suppose that $ay_1$,$\ldots$,$ay_{r_e}$ are the edges incident to $a$ that cross $e$ in $\widehat{D}$. We may assume that, for every $i$, $j$ with $1\leq i<j\leq r_e$, when we traverse $e$ from $u$ to $v$,  the crossing $x_i=e\cap ay_i$ precedes  the crossing $x_j=e\cap ay_j$. It is convenient to let $x_0=u$ and $x_{r_{e}+1}=v$.

   The edges of $G_0$ included in $D[e]$ are the segments of $D[e]-\{\sigma_1,\ldots,\sigma_{s_a}\}$. We enumerate these edges as $\tau_0^e$,$\dots$,$\tau_{r_e}^e$, so that  $\tau_i^e$ is included in the $x_ix_{i+1}$-arc of $D[e]$. Note that $\tau_1^e$ is incident to $u$, while $\tau_{r_e}$ is incident to $v$. 
   
   Let $T=\{\tau_i^e\;:\; e\in E(G)\; \text{and } 0\leq i\leq r_e\}$ be the set of edges of $G_0$. In Step 1, when we apply Corollary 2.2 to the edges in $D_0$, we obtain a partition $T_1\cup T_2$ of $T$. Instead of counting how many crossings are between the segments in 
   $\sigma_1,\ldots,\sigma_{s_a}$ and 
   the edges in one of the $T_i$'s when we redraw $T_i$ in page 2, we estimate the number $m$ of crossings   between $\sigma_1,\ldots,\sigma_{s_a}$ and the edges in $T$ when we draw all the crossing edges in $T$ in page 2. This will show that one of the two parts, either $T_1$ or $T_2$, can be drawn in page 2 creating at most $m/2$ crossings with the segments $\sigma_1,\ldots,\sigma_{s_a}$. To show our claim, it suffices to prove that $m\leq k-1$, and this is what we do next.
   
   For every point $p$ distinct from $a$ and contained in an edge $f$ incident to $a$, the {\em depth} $h(p)$ of $p$ is the number of crossings in $\widehat{D}$, contained in the open subarc of $f$ connecting $a$ to $p$. 
  When we redraw an edge $\tau_i^e$ in page 2, we can draw it so that it crosses at most $h(x_i)+h(x_{i+1})$  segments in $\sigma_1,\ldots,\sigma_{s_a}$. Such new drawing of $\tau_i^e$ is obtained  from letting the segment of $\tau_{i}^e$ near to $x_i$ follow the same dual path in $D$  that $x_i$ follows to reach $a$ via $ay_i$. Likewise the new end of $\tau_i^e$ near $x_{i+1}$ is defined. The new $\tau_{i}^e$ is obtained from connecting  the two {\em end segments} of $\tau_i^e$ inside the face of $D$ containing $a$. 
        
      Let $X(a)$ be the set of crossings involving edges incident to  $a$. For every $x\in X(a)$, there are precisely two elements in $T$, so  that when they are redrawn in page 2, one of its end segments mimics the arc between $x$ and $a$ inside the edge including $x$   and $a$. Each $v\in V(G)$ is incident to at most $s_v$ edges crossing in $D_0$. Then, for every $v\in V(G)$, there are are most $s_v$ edges in $T$, so  that when we redraw them in page 2, one of their ends mimics the dual path followed by the edge $\widehat{D}[xa]$. These two observations together imply that 
\begin{eqnarray*}
m &\leq&  
 \sum_{x\in X(a)} 2h(x)+\sum_{v\in V} h(v)s_v\\
 &<& 2\sum_{v\in V}\left(1+2+\ldots+(h(v)-1)\right)+\sqrt{k/2}\sum_{v\in V}h(v) \\
& \leq & \sum_{v\in V }{h(v)^2}+(\sqrt{k/2})s_a\leq \left(\sum_{v\in V} h(v)\right)^2+k/2\\
 &= & s_a^2+k/2< k\text{.}
\end{eqnarray*}
 Because $m$ is an integer less than $k$, $m\leq k-1$ as desired.
 \qed\end{proof}

At the end, we obtained a drawing $D_1$ of (a subdivision of) $G$ with less than $t/2 - \sqrt{t/8} + 1/8 + (k-1)/2$ crossings. Using (\ref{equation1}) it follows that
$$
  cr(D_{1})< \frac{1}{2}(k+\sqrt{k/2}) - \sqrt{t/8} + 1/8 + k/2 - 1/2
  = k + \sqrt{k/8} -\sqrt{t/8} - 3/8 < k\text{,}
$$
contradicting the fact that $cr(D_1)\geq cr(G)\geq k$.
\qed \end{proof}

\section{Exact values of the crossing number of the cone for simple graphs}
\label{Sec4}
	
In this section, we investigate the minimum crossing number of a cone, with the restriction of only considering simple graphs. We are interested in finding   the smallest integer  $f_{s}(k)$ for which there is a simple graph with crossing number at least $k$, whose cone has crossing number $f_{s}(k)$.
On one hand, we describe below a family of simple graphs  that shows that $f_{s}(k)\leq 2k$. Our best general lower bound is obtained from Theorem
\ref{thm:lowerbound}.    The main result in this section, Theorem \ref{thm:cr(CG)geq3},  help us to obtain exact values on $f_{s}(k)$ for cases when $k$ is small.

\begin{theorem}\label{thm:cr(CG)geq3}
Let $G$ be a simple graph with crossing number $k$. Then
\begin{itemize}
\item[(1)] if $k\geq 2$, then $cr(CG)\geq k+3$;
\item[(2)]  if $k\geq 4$, then $cr(CG)\geq k+4$;  and
\item[(3)] if $k\geq 5$, then $cr(CG)\geq k+5$.

\end{itemize}
\end{theorem}

Before proving Theorem 	\ref{thm:cr(CG)geq3}, we describe a family of examples that is used to find an upper bound for $f_s(k)$, that is exact for the values  $k=3,4,5$.
Given an integer $k\geq 3$, the graph $F_{k}$ (Figure \ref{fig:upperBound}) is obtained from two disjoint cycles $C_{1}=x_{0}\dots x_{k-1}x_0$ and $C_{2} = y_{0}\dots y_{2k-1}y_0$ by adding, for each $i=0,\dots,k-1$, the edges $x_{i}y_{2i-2}$, $x_{i}y_{2i-1}$, $x_{i}y_{2i}$, $x_{i}y_{2i+1}$ (where the indices of the vertices $y_j$ are taken modulo $2k$).
It is not hard to see that $F_{k}$ has crossing number $k$: a drawing with $k$ crossings is shown in Figure \ref{fig:upperBound}.
To show that $cr(F_{k})\geq k$, for $i\in \{ 0,\ldots,k-1\}$, consider  $L^i$, the $K_4$ induced by the vertices in   $\{x_{i},x_{i+1},y_{2i},y_{2i+1}\}$.
Every $L^i$ is a subgraph of a $K_5$ subdivision of $F_k$, thus, in an optimal drawing of $F_k$, at least one of the edges in $L^i$ is crossed. This only guarantees that $\CR(F_k)\geq k/2$, as two edges from distinct $L^i$'s might be crossed. However,  if an edge from $L^i$ crosses an edge $e_j$ from some other $L^j$, then $F_k-e_j$ has a $K_5$ subdivision including $L^i$, exhibiting a new crossing in some edge in $L^i$. Therefore, every $L^i$ either has a crossing not involving an edge in another $L^j$, or there are least two crossings involving edges in $L^i$. This shows that $\CR(F_k)\geq k$.


\begin{figure}[ht]
	\begin{center}
		\begin{tikzpicture}[scale=2]
			\tikzstyle{every node}=[minimum width=0pt, inner sep=1pt, circle]
			
			\draw (240:1) arc (-120:210:1) ;
			\draw [domain=-105:195] plot ({cos(\x)/2}, {sin(\x)/2});
			
			\foreach \x in {0, 1, 2, 3, 4, 5, 6, 7, 8}
			{
				\draw (30*\x - 30:1) node (u\x) [draw, fill= white, label=30*\x-30:{\tiny $y_{\x}$}] {};
			}
			\draw (240:1) node (u9) [draw, fill= white, label=240:{\tiny $y_{2k-3}$}] {};
			\draw (270:1) node (u10) [draw, fill= white, label=270:{\tiny $y_{2k-2}$}] {};
			\draw (300:1) node (u11) [draw, fill= white, label=-60:{\tiny $y_{2k-1}$}] {};
			
			\foreach \x in {0, 1, 2, 3, 4}
			{
				\draw (60*\x - 45:0.5) node (v\x) [draw, fill= white, label=60*\x+120:{\tiny $x_{\x}$}] {};
			}
			\draw (255:0.5) node (v5) [draw, fill= white, label=60:{\tiny $x_{k-1}$}] {};
			
			\foreach \x in {1, 2, 3}
			{
				\pgfmathtruncatemacro{\i}{2*\x-2}
				\pgfmathtruncatemacro{\j}{2*\x-1}
				\pgfmathtruncatemacro{\k}{2*\x}
				\pgfmathtruncatemacro{\l}{2*\x+1}
				
				\foreach \y in {\i, \j, \k, \l}
				{
					\draw (u\y) -- (v\x);
				}
			}	
			\draw (u10) -- (v0);
			\draw (u11) -- (v0);
			\draw (u0) -- (v0);
			\draw (u1) -- (v0);
			\draw (u6) -- (v4);
			\draw (u7) -- (v4);
			\draw (u8) -- (v4);
			\draw (u9) -- (v5);
			\draw (u10) -- (v5);
			\draw (u11) -- (v5);

			\foreach \x in {1, 2, 3}
			{
				\draw (5*\x + 215:0.7) node [draw, fill= black, inner sep=.3pt] {};
			}
		\end{tikzpicture}
	\end{center}
	\caption{The graph $F_k$.}		
	\label{fig:upperBound}
\end{figure}
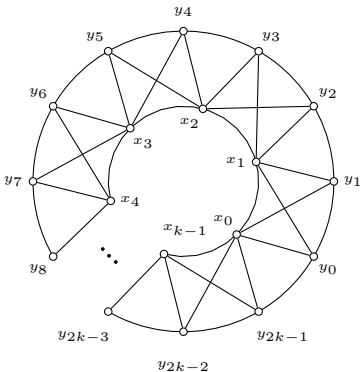

The graph shown in Figure \ref{fig:GwithCR2}  has crossing number 2, and  its cone has crossing number 5.
This shows that $f_{s}(2)\leq 5$.
On the other hand, $F_{3}$, $F_{4}$, and $F_5$ serve as examples to show that
$f_s(k)\leq 2k$ for $k=3,4,5$. These bounds are tight for $2\leq k\leq 5$ by Theorem \ref{thm:cr(CG)geq3}.
	
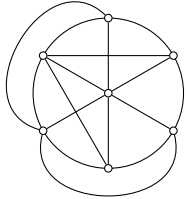
\begin{figure}[h]
	\begin{center}
	\begin{tikzpicture}[scale=1]
    		\tikzstyle{every node}=[minimum width=0pt, inner sep=1pt, circle]
    		\draw (0,0) circle (1);
    		\draw (0,0) node (0) [draw=black, fill=white] {};
    		\foreach \x in {1, 2, 3, 4, 5, 6}
    		{
    			\draw (60*\x - 30:1) node (\x) [draw=black,fill=white] {};
    			\draw (0) edge (\x);
    		}
    		\draw (1) -- (3);
    		\draw (3) -- (5);
    		\draw (2) .. controls (115:2) and (185:2) .. (4);
    		\draw (4) .. controls (235:2) and (305:2) .. (6);
    	\end{tikzpicture}
	\end{center}
	\caption{A graph with crossing number 2 whose cone has crossing number 5.}		
	\label{fig:GwithCR2}
\end{figure}

\begin{proof}[Proof of Theorem \ref{thm:cr(CG)geq3}]
Suppose $G$ is a graph with $cr(G)=k$. Let $\widehat{D}$ be  an optimal drawing of the cone $CG$, $D$ its restriction to $G$, and $F_{a}$ be the face of $D$ containing the apex $a$. The vertices of $G$ incident to $F_{a}$ are the {\em planar neighbors} of $a$.

Assume that $k\geq 2$, and suppose $\widehat{D}$ has exactly  $k+t$ crossings. Theorem \ref{thm:lowerbound} guarantees that $t\geq 1$. Since each edge from $a$ to a non-planar-neighbor introduce at least one crossing,  the apex $a$ has either 0, 1,  2, 3  or 4 non-planar neighbors (if $a$ has more than 4 non-planar neighbours, then any of the items in Theorem \ref{thm:cr(CG)geq3} is satisfied).

We start by assuming that $a$ has no non-planar neighbors.  In this  case, $D$ is a 1-page drawing of $G$. Corollary \ref{cor:1-2-page} implies that we can
obtain a new drawing of $G$ with less than $(k+t)/2$ crossings.
Thus $(k+t)/2> cr(G)=k$, which implies that $t\geq k+1$. In any case
of the theorem, this implies the conclusion, thus we may now assume that $a$ has at least one and at most $t$ non-planar neighbors.

{\bf (1)} Let us now assume that $k\geq 2$ and $t\leq 2$.

Suppose $a$ has exactly one non-planar neighbor $u$. Then $cr(D)$ has at most $k+1$ crossings. At least one edge incident to $u$ is crossed in $D$, otherwise, all the crossed edges have ends in $F_{a}$, and
using Corollary \ref{cor:1-2-page}, we obtain a drawing of $G$ with less than $(k+1)/2$ crossings, contradicting that $cr(G)=k$. If at least two crossings in $D$ involve edges incident to $u$, or if $D$ has $k$ crossings, then by redrawing $u$ in $F_{a}$, and adding all the edges to its neighbors without creating any crossings, we obtain a drawing of $G$ with less than $k$ crossings. Therefore $D$ has $k+1$ crossings, and exactly $k$ of them involve edges not incident to $u$. Again, we apply Corollary \ref{cor:1-2-page} to  obtain a drawing of $G$ with at most $\frac{1}{2}(k-1)<k$ crossings (this time we are more careful by
setting our two pages in such way that
  the edge  not incident to $u$ that crosses an edge incident to $u$ is redrawn in the page contained in $F_{a}$).

Finally, suppose $a$ has exactly two distinct non-planar neighbors $u$ and $v$. Then, $\widehat{D}$ has $k+2$ crossings; $D$ has $k$ crossings, and the edges $au$, $av$ are crossed exactly once. Notice that any crossed edge in $D$ is incident to either $u$ or $v$; otherwise, we can redraw such edge inside $F_{a}$, obtaining a drawing of $G$ with less than $k$ crossings.  Redraw $v$ in $\widehat{D}[a]$ (where $\widehat{D}[a]$ denotes the point representing $a$   in $\widehat{D}$); draw  the edge $uv$  (if it exists in $G$)  as the arc $\widehat{D}[au]$, and draw  the edges   from $v$ to its neighbors distinct from $u$, inside $F_{a}$ without creating  new crossings. Since every crossing in $D$ involves an edge incident with $v$, we obtain a drawing  of $G$ with at most one crossing, a contradiction.

{\bf (2)} Now,  suppose  that  $k\geq 4$ and that $t=3$.

The case when the  apex $a$ has only one non-planar neighbor $u$ is similar to the above.
If at least three crossings in $D$ involve edges incident with $u$, then by redrawing $u$ and the edges incident to $u$ in $F_{a}$, we obtain a drawing with less than $k$ crossings, a contradiction. Thus, at  most two crossings involve edges incident to $u$. We redraw the remaining crossed edges according to Corollary \ref{cor:1-2-page} (if there is an edge $e$ that crosses an edge incident to $u$, in order  to remove an extra crossing,
we may choose this new drawing so that $e$ is redrawn in the page contained in $F_{a}$). If  2 crossings involve edges incident to $u$, then the obtained  drawing has at most $\frac{k}{2}+1$ crossings, where the  $+1$ comes from the fact that $e$ was drawn in the page contained in $F_{a}$. If at most one of the edges at $u$ is crossed, then the new drawing has at most $(k+1)/2$ crossings. In any case, since $k\geq 4$, the new drawing has less than $k$ crossings, a contradiction.


Let us now consider the case when the apex has two non-planar neighbors $u$ and $v$.  In this case, the drawing $D$ has either $k$ or $k+1$ crossings, and one of $\{au, av\}$, say $au$, is crossed only once. Let $L$ be the set of crossed edges in $D$
that are not incident to $u$ or $v$. Suppose there are at least two crossings involving only edges in $L$. Then, either there are two edges in $L$ that do not cross, or $L$ has an edge $e$ that crosses two other edges in $L$.
 In the former case, we redraw such pair of edges in $F_{a}$; in the latter case, we redraw $e$ in $F_{a}$. Any of these modifications yield a drawing with less than $k$ crossings. Thus, we may assume that at most one crossing in $D$ involves two edges not incident to  $u$ or $v$.
Redraw $v$ in $\widehat{D}[a]$; draw the edge $vu$ (if such edge exists in $G$) as $\widehat{D}[au]$; and the remaining edges from $v$
to its neighbors distinct from $u$ without creating  new crossings. The new drawing of $G$ has at most two crossings: possibly one in $\widehat{D}[av]$ and another  between edges in $L$, a contradiction.

Finally suppose that the apex $a$ has three non-planar neighbors $u$, $v$, $w$.
In this case $D$ has precisely $k$ crossings, and the edges $au$, $av$, $aw$ are crossed exactly once. Observe that any crossed edge in $D$
is incident to one of $\{u,v,w\}$, otherwise we can redraw such edge in $F_{a}$, obtaining a drawing of $G$  with less than $k$ crossings.

Let $H$ be the graph induced by $\{u,v,w\}$. If, for $x\in \{u,v,w\}$, $d_{H}(x)$ denotes the degree
of $x$ in $H$, then at most  $d_{H}(x)$ crossings involve edges at $x$.
Otherwise, by redrawing $x$ in $\widehat{D}[a]$; drawing the edges from $x$ to its neighbors in $H$ by using the respective edges from $a$; and, by drawing the remaining edges at $x$ in $F_{a}$ without creating new crossings, we obtain a drawing of $G$ with
less than $k$ crossings. So for each vertex $x\in \{u,v,w\}$, there are at most two crossings involving edges at $x$. Hence $D$ has at most three crossings, a contradiction.

{\bf (3)} The  proof can be found in Apendix \ref{app:1}, and it will be included in the full version of the paper.
	\qed \end{proof}

\section{Asymptotics for simple graphs} \label{sect:asymptotics}

Lastly, we try to understand the behaviour of $f_s(k)$ when $k$ is large.
The important part is the increase of the crossing number after adding the apex, thus we define
$$
    \phi_s(k) = f_s(k)-k.
$$
We have proved that $\phi(k)=f(k)-k\geq \frac{1}{2}k^{1/2}$. The term $k^{1/2}$ is asymptotically tight in the case when we allow multiple edges. However, it is unclear how large  $\phi_s(k)$ is. This question is treated next.

\begin{theorem}\label{thm:O(k34) bound}
	$\phi_s(k)=O(k^{3/4})$.
\end{theorem}
\begin{proof}
	Let us consider a positive integer $k$ and let $n$ be the smallest integer such that $\CR(K_n)\geq k$. Then $G=K_n$ has a crossing number at least $k$ and its cone is $K_{n+1}$.
	
	To find an upper bound for $\CR(K_{n+1})$ in terms of $\CR(K_n)$, start with a drawing of $K_n$ with $\CR(K_n)$ crossings. Then clone a vertex, that is, place a new vertex very close to an original vertex, and draw the new edges along the original edges. Each edge incident to the new vertex cross $O(n^2)$ edges, thus the obtained drawing has $\CR(K_n)+O(n^3)$ crossings. Therefore 
	
	$$\phi_s(k)\leq \CR(K_{n+1})-\CR(K_n)\leq O(n^3)\text{.} $$
	
	It is known \cite{KPS07} that
	$$
	\frac{3}{10}\binom{n}{4} \le \CR(K_n) \le \frac{3}{8}\binom{n}{4}.
	$$
	(The constant $3/10$ in the lower bound has been recently improved to $0.32025$, see \cite{KPS07} for more information.) Then $\phi_s(k)=O(n^3)=O(k^{3/4})$.
\qed\end{proof}

The \emph{Harary-Hill Conjecture} \cite{HH63}
states that
$$
   \CR(K_n) = \left\{
                \begin{array}{ll}
                  \tfrac{1}{64}n(n-2)^2(n-4), & \hbox{$n$ is even;} \\[2mm]
                  \tfrac{1}{64}(n-1)^2(n-3)^2, & \hbox{$n$ is odd.}
                \end{array}
              \right.
$$

\begin{proposition}\label{thm:O(k34) bound assuming HHC}
If the Harary-Hill conjecture holds, then
\begin{equation*}
   \phi_s(k) \le \sqrt{2}\, k^{3/4}(1+o(1)).\label{eq:under HHC assumption}
\end{equation*}
\end{proposition}

\begin{proof}
As in the proof of Theorem \ref{thm:O(k34) bound}, but with a slight twist for added precision, we take $n$ such that $\CR(K_{n-1}) < k \leq \CR(K_n)$.
We also take $n_1$ such that for $k_1 = k-\CR(K_{n-1})$ we have $\CR(K_{n_1-1}) < k_1 \le \CR(K_{n_1})$. Let $G = K_{n-1} \cup K_{n_1}$. Then $\CR(G) = \CR(K_{n-1}) + \CR(K_{n_1}) \ge k$ and $\CR(CG) = \CR(K_{n}) + \CR(K_{n_1+1})$. Therefore,
\begin{eqnarray*}
	\phi_s(k) & \leq & \CR(K_n)+\CR(K_{n_1+1})-\CR(K_{n-1})-\CR(K_{n_1})\\
	& \leq & \CR(K_n)-\CR(K_{n-1})+\CR(K_{n_1+1})-\CR(K_{n_1-1})\text{.}
\end{eqnarray*}

By inserting the values for the crossing number from the Harary-Hill Conjecture, we obtain (the calculation given is for odd $n$ and odd $n_1$, it is similar when $n$ or $n_1$ is even):
$$
  \CR(K_n) - \CR(K_{n-1}) =\tfrac{1}{64}((n-1)^2(n-3)^2-(n-1)(n-3)^2(n-5))
= \tfrac{1}{16}n^3(1+o(1))
$$

\noindent and
\begin{eqnarray*}
  \CR(K_{n_1+1}) - \CR(K_{n_1-1}) &=&
  \tfrac{1}{64}((n_1+1)(n_1-1)^2(n_1-3)\\
  & &-(n_1-1)(n_1-3)^2(n_1-5))\\
  &=& \tfrac{1}{8}n_1^3(1+o(1)).
\end{eqnarray*}
Noticing that $k=\tfrac{1}{64}n^4(1+o(1))$ and $k_1 = \tfrac{1}{64}n_1^4(1+o(1)) = O(n^3)$ because $k_1\leq cr(K_n)-cr(K_{n-1})$, we conclude that $n_1^3 = O(n^{9/4}) = o(k^{3/4})$ and henceforth
$$
   \phi_s(k) \le \tfrac{1}{16}n^3(1+o(1)) + \tfrac{1}{8}n_1^3(1+o(1))
             = \sqrt{2}\, k^{3/4}(1+o(1)).
$$
\qed \end{proof}

The above proof works even under a weaker hypothesis that $\CR(K_n) = \alpha n^4 + \beta n^3(1+o(1))$, where $\alpha$ and $\beta$ are constants. This would imply that $\phi_s(k) = O(k^{3/4})$. Our conjecture is that (\ref{eq:under HHC assumption}) gives the precise asymptotics.

\begin{conjecture}\label{conj:phi_s}
$\phi_s(k) = \sqrt{2}\, k^{3/4} (1+o(1))$.
\end{conjecture}

A reviewer noted that this asymptotic is matched when the graph we are considering is dense.
\begin{remark}
	Let $G$ be a graph with $n$ vertices, $m$ edges, $\CR(G)=k$ and such that $m\geq 4n$. If $m=\Omega (n^2)$, then $\CR(CG)\geq k+\Omega(k^{3/4}) $.
\end{remark}
The details will be provided in the full version.
%
%
%
%
\subsection*{Summary}

To put the results of this paper into context, let us overview some of the motivation and some of directions for future work. 
The starting point of this paper was an attempt to understand Albertson's conjecture. The results of the paper (and their proofs) show that the crossing number behavior when adding an apex vertex is intimately related to 1-page drawings, but the exact relationship is quite subtle.
There is some evidence that the minimal increase of the crossing number when an apex is added should be achieved with very dense graphs, close to the complete graphs. Our Conjecture \ref{conj:phi_s} entails this problem. Although very dense graphs have fewer vertices than sparser graphs with the same crossing number and thus need fewer connections to be made from the apex to their vertices, their near optimal drawings are far from 1-page drawings and therefore more crossings are needed.
The full understanding of this antinomy would shed new light on the Harary-Hill conjecture.

Finally, it is worth pointing out that neither exact nor approximation algorithm is known for computing the crossing number of graphs of bounded tree-width. Adding an apex to a graph increases the tree-width of the graph by 1, thus understanding the crossing number of the cone is an important special case that would need to be understood before devising an algorithm for general graphs of bounded tree-width.

\newpage

\appendix

\section{Appendix: Proof of part (3) of Theorem \ref{thm:cr(CG)geq3} }\label{app:1}
In this section we complete the proof of Theorem \ref{thm:cr(CG)geq3} by showing that if $cr(G)=k$ and  $k\geq 5$, then $cr(CG)\geq k+5$.

\begin{proof}

 Suppose $G$ is a graph with $cr(G)=k$. Let $\widehat{D}$ be  an optimal drawing of the cone $CG$, $D$ its restriction to $G$, and $F_{a}$ be the face of $D$ containing the apex $a$. The vertices  of $G$ incident to $F_{a}$ are the {\em planar neighbors} of $a$. If $u$ is a non-planar neighbor of $a$, then we let $s_{au}$ to be the number crossings involving  the edge $au$ in $\widehat{D}$.


Assume that $k\geq 5$, and suppose $\widehat{D}$ has less than $k+5$ crossings. By part (2) of Theorem 4.1  we know that $cr(\widehat{D})=k+4$. Suppose that $cr(D)=k+t$.
Let $N$ denote the set of non-planar neighbors of $a$.
As we did before, we split into cases depending on the size of $N$.

If the apex $a$ has only planar neighbors, then, applying Corollary \ref{cor:1-2-page}, we can obtain a drawing with less than  $\tfrac{k+4}{2}<k$ crossings, a contradiction. Thus $|N|\geq 1$.

We need the following observation.

\begin{claim}\label{help}
	If $u\in N$,  then the following holds:
	\begin{itemize}
		\item[(i)] At most 
		$4-s_{au}$
		crossings in $D$ involve edges incident to $u$.
		
		\item[(ii)] The number of crossing in which both edges involved are incident to some vertex in $N$ is at most
		$2|N|-\lceil s/2 \rceil$.
		
	\end{itemize}
\end{claim}
\begin{proof}
	(i) If  there are more than 
	$4-s_{au}$
	crossings involving edges at $u$, then we redraw  $u$ in the place of $\widehat{D}[a]$,  join $u$ to its neighbors using the corresponding edges from $a$ to $V(G)$. This is a drawing with less than
	$$k+4-s-(4-s_{au})+s-s_{au}=k$$
	crossings, a contradiction.
	
	(ii) From   (i), we know that for each $u\in N$, there are at most  
	$4-s_{au}$
	crossings in $D$ involving  edges at $u$. Let us count the number of pairs $(u,x)$ where $u\in N$ and $x$ is a crossing involving an edge incident with $u$. By (i), the number of such pairs is at most
	
	\begin{eqnarray}
	\label{eq:claim4.1.3} 
	\sum_{u\in N}(4-s_{au}) & = & 4|N|-s\text{.}
	\end{eqnarray}
	This in particular implies (ii).
	\qed
\end{proof}

\begin{Case}
	The apex $a$ has exactly one planar neighbor.
\end{Case}
From Item (i) in the previous  Claim, we know that  there are at most $4-s\leq 3$ crossings involving edges incident with $u\in N$. So at least $k$ crossings  involve crossing pairs that are not incident to  $u$. We apply Corollary \ref{cor:1-2-page} to redraw some  crossing edges not incident with $u$  in $F_a$, and we are careful by  choosing our two pages  so that we draw one edge crossing an edge at $u$ (if such edge exists) in $F_a$ to remove an extra crossing. Then we
obtained a drawing with at most
$k+4-s-(\frac{k}{2}+\frac{1}{2})-1<k $  crossings.\\

\begin{Case}
	The apex $a$ has two non-planar neighbors.
\end{Case}

Let us define $t=4-s$ for brevity. Then   $0\leq t\leq 2$. Let $X$ be the set of crossings involving an edge not incident to a  vertex in $N$, and $E_X$ be the set of crossing edges not incident to any vertex in $N$. We claim that $|X|\leq t$.

This is easy to see when   $t=0$, as if $|X|\geq 1$, then  there is an edge in $E_X$ that we can  redraw  in $F_a$ to obtain a drawing with less than $k$ crossings.
Suppose that $t=1$. If $|X|\geq 2$, then,   either there is an edge   $e\in E_X$ including two crossings, or there is a pair of edges in $E_X$ that are not crossed. In the former case we redraw $e$ in $F_a$, in the latter we redraw the pair in $F_a$, to obtain a drawing with less than $k$ crossings.

Finally, suppose that $t=2$. Any  edge in $E_X$ is involved in at most 2 crossings, otherwise we could redraw it in $F_a$ to obtain a drawing with less than $k$ crossings.
If $|X|\geq 3$, then, either there is an edge $e\in E_X$ crossed  twice and an edge $f\in E_X$ not crossing $e$, or every edge in $E_X$ is crossed at most once. In the former case, we redraw $e$ and $f$ in $F_a$, in the latter, for each crossing in $X$ we pick an edge $E_X$ involved in the crossing and redraw it in $F_a$. In any case we obtain a drawing of $G$ with less than $k$ crossings. Therefore $|X|\leq 2$.

In any case we know that $|X|\leq t\leq 2$, and by Item (ii),  the number of crossings in $D$ is at most

$$ 2+t/2+|X|\leq 2+3t/2\leq 5\text{.} $$

Since $cr(G)\geq 5$, we have that $cr(D)=5$, $|X|=t=2$, and  $E_X\neq \emptyset$. However, if we redraw any edge from $E_X$ in $F_a$, we obtain a drawing of $G$ with less than $5$ crossings.

\begin{Case}
	The apex $a$ has three non-planar neighbors.
\end{Case}
In this case  $cr(D)$ is either $k$ or $k+1$, so $t=0$ or $t=1$. The argument given in the previous case shows that
there are at most $t$ crossings involving an edge not incident to a vertex in  $N$. Using Item (ii), we know that  $D$ has at most $4+t/2+|X|$ crossings, where $X$ is defined as in the previous case. Since $cr(D)\geq 5$, this shows that
$D$ has exactly 5 crossings, and that $|X|=1$. In particular $cr(D)=k$ and thus $t=0$, which contradicts that $|X|\leq t$.

\begin{Case}
	The apex has four non-planar neighbors.
\end{Case}

In this case $s=4$, $cr(D)=k$ and $s_{ay}=1$ for every $y\in N$. Let $N=\{u,v,w,x\}$. Note that each  crossing edge is incident to a vertex in $N$. By Item (i), there are at most 3 crossings involving edges incident to a fixed vertex in $N$, and by (ii), $cr(D)\leq 6$. Moreover,  the count (\ref{eq:claim4.1.3})  in the proof of (ii) shows that $cr(D)=5$ if there is an edge with both ends in $N$ that  is involved in a crossing. 

Let $H$ be the graph induced by $N$.
We will split into two cases depending on whether $D[H]$ is a crossing $K_4$ or not.


\begin{subcase} $D[H]$ is not a crossing $K_4$.
\end{subcase}

If $H=K_4$, then $D[H]$ is a planar $K_4$. This implies that there is a 3-cycle composed of vertices in $N$, separating a fourth vertex in $N$ from $a$, and this contradicts that $s_{ay}=1$ for every $y\in N$. Therefore, there is a pair of vertices in $H$, say $u$ and $v$, with $uv\notin E(G)$.

If, for $y \in N$, $d_{H}(y)$ denotes the degree
of $y$ in $H$, then at most  $d_{H}(y)$ crossings involve edges at $a$.
Otherwise, by redrawing $y$ in $\widehat{D}[a]$; drawing the edges from $y$ to its neighbors in $H$ by using the respective edges from $a$; and, by drawing the remaining edges at $y$ in $F_{a}$ without creating new crossings, we obtain a drawing of $G$ with
less than $k$ crossings. Since $d_{H}(u), d_H(v)\leq 2$ and $d_{H}(w)$, $d_H(x)\leq 3$, $cr(D)$ is at most
$\left(\sum_{y\in N} d_H(y) \right)/2=5$. Because $cr(D)\geq 5$,  this implies that $cr(D)=5$, $d_H(u)=d_H(v)=2$ and that $d_H(w)=d_H(x)=3$. This also shows that for each $y\in N$, the number of crossings involving edges at $y$ is exactly $d_H(y)$. Also note that none of the edges in $H$ are crossed, as otherwise, at least three of the four vertices involved in some crossing belong to $N$, and a refined version of our previous counting would exhibit  that $D$ has at most 4  crossings.

Let $H'$ be the drawing induced by $N\cup\{a\}$. Our previous observations imply that the drawing of $D[H']$ is isomorphic to the drawing of the cone of a planarly drawn $K_4$ minus one edge, where the apex is drawn in the face bounded by the 4-cycle of $K_4-e$, and the edges incident to the apex connect directly to the boundary of the 4-cycle. Moreover, the only crossings of $H'$ in $\widehat{D}$, are those between the edges at $a$ and the boundary of $F_a$ in $D$. This restricted drawing of $H'$ implies that  the ends of a crossing pair of edges have exactly one element in $\{u,v\}$, exactly one element in $\{w,x\}$, and none of the two edges has both ends in $N$. However, this is not possible, as there are 4 crossings involving edges incident to one of $u$ or $v$, while there are 6 crossings involving edges incident to one of $w$ or $x$.

\begin{subcase}
	$D[H]$ is a $K_4$ with a crossing.
\end{subcase}

Suppose that $uv$ and $wx$ is the crossing pair in $D[H]$, and  that $\times$ is the crossing between $uv$ and $wx$.
Following the same argument given  in   the previous case, it is easy to see that  for every $y\in N$, there are exactly 2 crossings distinct from $\times$ involving edges at $y$;  $cr(D)=5$; if $H'$ is graph is induced by $N\cup\{a\}$, then its  drawing  $D[H']$ is isomorphic to the drawing of the cone of a crossing $K_4$, where the apex is drawn in the face bounded by the 4-cycle of the $K_4$, and the edges incident to the apex connect directly to the boundary of the 4-cycle; and,  the only crossings of $\widehat{D}$ in $H'$,  distinct from $\times$, are those between the edges at $a$ and the boundary of $F_a$ in $D$.

The restrictions on $H'$ show that  the ends of a crossing pair of edges distinct from $uv$ and $wx$, have exactly one element in $\{u,v\}$, exactly one element in $\{w,x\}$, and none of the two edges has both ends in $N$.

The boundary walk of $F_a$ contains a cycle $C$ that in  $\widehat{D}$ separates $a$ from $N$.
There are two internally disjoint subarcs $\alpha$ and $\beta$ of $D[C]$ connecting the crossings between $D[C]$ and each of $aw$ and $ax$. We label $\alpha$ and $\beta$ so that $\alpha$ includes the crossing between $D[C]$ and $au$,  and $\beta$ includes the crossing between $D[C]$ and $av$.
The restrictions imposed by the crossings in $H'$ imply that all the neighbors of $u$  not in $N$ are contained in $\alpha$, and likewise the neighbors of $v$ not in $N$ are contained in $\beta$.

We obtain a drawing of $G$ with 4 crossings as follows. Redraw $u$ in the place of $\widehat{D}[a]$; join $a$ to each of $w$ and $x$ using the corresponding edges from $a$ to each of $w$ and $x$. Draw the edges from $u$ to its  neighbors not in $N$ without creating new crossings. Now redraw $v$ near $u$ in the face bounded by $\beta$ and the two segments of the new  $uw$, $ux$ edges. Connect $v$ to each of  $w$ and $x$ by following arcs near the new  $uw$, $ux$ edges. Connect $v$ to the rest of its neighbors without creating new crossings. Since $cr(G)\geq 5$ and this  drawing has 4 crossings, this is a contradiction.

In any case we obtained a contradiction.   Thus $cr(CG)\geq k+5$ when $k\geq 5$.
\qed \end{proof}

\end{document}